\documentclass{article}

\usepackage[dvips]{graphicx}
\input xy
\xyoption{all}
\usepackage{amssymb, amsmath, amsfonts, epsfig, color,amscd,amsthm,setspace,lineno}

\title{A Lower Bound for the Number of Group Actions on a Compact Riemann Surface}

\author{James W. Anderson\\School of Mathematics \\University of Southampton\\Southampton SO17 1BJ, UK\\j.w.anderson@soton.ac.uk\\ \\Aaron Wootton\\Department of Mathematics\\University of Portland\\Portland, OR 97203, USA\\wootton@up.edu}

\newtheorem{theorem}{Theorem}[section]

\newtheorem{conjecture}[theorem]{Conjecture}
\newtheorem{corollary}[theorem]{Corollary}
\newtheorem{lemma}[theorem]{Lemma}
\newtheorem{proposition}[theorem]{Proposition}

\newtheorem{definition}[theorem]{Definition}

\begin{document}

\maketitle

\begin{abstract}
We prove that the number of distinct group actions on compact Riemann surfaces of a fixed genus $\sigma \geq 2$ is at least quadratic in $\sigma$.  We do this through the introduction of a coarse signature space, the space $\mathcal{K}_\sigma$ of {\em skeletal signatures} of group actions on compact Riemann surfaces of genus $\sigma$.  We discuss the basic properties of $\mathcal{K}_\sigma$ and present a full conjectural description. 
\end{abstract}

\section{Introduction}
\label{introduction}

The purpose of this note is to prove (Theorem \ref{thm-main}) that there exist at least $\frac{1}{4} (k_\sigma+1)(k_\sigma+3)$ distinct actions of groups of conformal automorphisms on compact Riemann surfaces of a fixed genus $\sigma \geq 2$, where $k_\sigma = \lfloor \frac{\sigma}{3}\rfloor$.  We start by putting this result in context. 

Let $X$ be a compact Riemann surface of genus $\sigma \ge 2$ and let $G$ be a group of conformal automorphisms acting on $X$.  The signature of the action is the tuple $(h; n_{1}, \ldots, n_{r})$, where the quotient space $X/G$ has genus $h$ and the quotient map $\pi \colon X\rightarrow X/G$ is branched over $r$ points with orders $n_1,\ldots, n_r$.    We work here with the definition that the actions of groups $G_{1}$ and $G_{2}$ on compact Riemann surfaces $X_{1}$ and $X_{2}$ (of the same genus $\sigma$) are equivalent if $G_1$ is isomorphic to $G_{2}$ and if the signatures of the actions of $G_{1}$ on $X_1$ and of $G_{2}$ on $X_2$ are equal.  

The {\em counting problem} we are interested in is to count the number of equivalence classes of such actions.    Though an interesting question in its own right, one of the primary motivations for our work comes from the closely related problem of counting conjugacy classes of finite subgroups of the mapping class group ${\rm MCG}(S)$ for a closed orientable surface $S$ of genus $\sigma$.   Specifically, a consequence of the solution to the Nielsen Realization Problem (see Kerckhoff \cite{kerckhoff}) is that there is a one-to-one correspondence between the conjugacy classes of finite subgroups of ${\rm MCG}(S)$ and the number of distinct group actions on $S$ (up to isotopy). Though the equivalence of group actions we are considering is coarser, it still provides a lower bound for the number of such actions and hence for the number of distinct conjugacy classes of finite subgroups of ${\rm MCG}(S)$.

If $G$ acts on $X$ with a given signature $(h; n_{1}, \ldots, n_{r})$, then the Riemann-Hurwitz formula (see Section \ref{preliminaries}) is satisfied.   For a fixed $\sigma$ and a fixed order $|G|$ of $G$, there are then only finitely many solutions to the Riemann-Hurwitz formula.  This gives the very crude estimate that the number of groups that can act on some such $X$ is finite, being the number of groups of order at most the Hurwitz bound of $84(\sigma -1)$ multiplied by the number of possible signatures satisfying the Riemann-Hurwitz formula. However, this estimate is unreasonably crude. 

The main tool we use is the space $\mathcal{K}_\sigma$ of {\em skeletal signatures} for actions of groups on Riemann surfaces of genus $\sigma$.  A skeletal signature of an action of a group $G$ on a Riemann surface $X$ is the ordered pair $(h_0, r_0)$ where $h_0$ is the genus of the quotient $X/G$ and $r_0$ is the number of branch points for the covering $X\rightarrow X/G$.  We provide a detailed discussion of the basic properties of $\mathcal{K}_\sigma$ in Section \ref{properties}, and provide a complete conjectural picture.   The proof of Theorem \ref{thm-main} proceeds by showing that there are quadratically many (in $\sigma$) different skeletal signatures corresponding to actions of the cyclic group $C_4$ of order $4$ on closed Riemann surfaces of genus $\sigma$.

Actions of finite groups on Riemann surfaces have been extensively studied, and we do not provide here a full survey of what is known. There are a number of previous and current results closely related to this project. One approach to the counting problem is to fix a genus and attempt to classify all groups which can act on a surface of that genus.  For example, the numbers of distinct topological group actions on surfaces of genus $2$ and $3$ were determined in Broughton \cite{Bro1}, and there are many other results for other small genera, see for example Bogopol'ski\u\i \cite{Bog} and Kuribayashi and Kimura \cite{Kur}. 

More recently, dramatically extending these results, Breuer \cite{Breu} determined the number of distinct group actions for each surface of genus $\sigma$ for $2\leq \sigma \leq 48$. Though these results are extremely impressive, the difficulty of enumerating the total number of distinct group actions on a surface of a fixed genus increases quickly as the genus increases. In particular, there seems little hope that one would be able to give an exact answer to the counting question. 

A different approach to classification of group actions is to instead consider the problem of classifying special families of groups. For example, the number of distinct cyclic group actions of prime order up to topological equivalence on a surface of genus $\sigma$ was determined in Harvey \cite{Har}, and methods to derive similar results for elementary abelian groups were given in Broughton and Wootton \cite{Bro2}.  Many other results exist for other families of groups, for example Tyszkowska \cite{Tys1}. 

We would like to thank the referee for their careful reading of the paper.  The first author would like to thank the University of Portland and the second author would like to thank the University of Southampton for their hospitality during the various visits made during the work on this paper. 

\section{Preliminaries}
\label{preliminaries}

In this Section, we introduce the necessary preliminaries.  

We begin with some notation.  Let $(x)$ be the result of rounding the real number $x >0$ to the nearest integer.   For a natural number $n$, let $C_n$ be the cyclic group of order $n$.   For an arbitrary group $G$, let $e_G$ be its identity element and let $|G|$ denote the order of $G$.

\begin{definition}
Let $G$ be a finite group and let $X$ be a compact Riemann surface of genus $\sigma \geq 2$. We say that $G$ {\em acts on $X$ with signature} $(h; n_1, \ldots, n_r)$ if the elements of $G$ are conformal automorphisms of $X$, the quotient space $X/G$ has genus $h$ and the quotient map $\pi \colon X\rightarrow X/G$ is branched over $r$ points with orders $n_1,\ldots, n_r$. 
\end{definition}

There is an alternative notation for signature that we will have occasion to use, in which we organize the branch points by grouping them together by order.  In this case, we say that $G$ {\em acts on $X$ with signature} $(h; [n_1,t_1],\ldots, [n_s,t_s]) $ if the quotient $X/G$ has genus $h$ and the quotient map $\pi \colon X\rightarrow X/G$ is branched over $t_j$ points with order $n_j$ for each $1\le j \le s$.

It is standard that if $G$ acts on $X$ with signature $(h; n_1, \ldots, n_r)$, then the {\em Riemann-Hurwitz formula} is satisfied:
$$\sigma -1=|G|(h-1) +\frac{|G|}{2} \sum_{j=1}^{r} \bigg( 1-\frac{1}{n_{j}} \bigg).$$ 

The natural question that arises is then to ask, if a signature satisfies the Riemann-Hurwitz formula for a given $\sigma\ge 2$, what additional information is needed to conclude that the signature arises from the action of a group $G$ on some compact Riemann surface of genus $\sigma$.  For this, we need the following definition.

\begin{definition}
Let $G$ be a finite group.  A vector $(a_{1},b_{1},a_{2},b_{2},\dots ,a_{n},b_{n},c_{1},\dots ,c_{r})$ of elements of $G$ is an {\em $(h;n_1,\ldots, n_r)$-generating vector for $G$} if the following hold:
\begin{enumerate}
\item $G=\langle a_{1},b_{1},a_{2},b_{2},\dots ,a_n,b_n,c_{1},\dots ,c_{r} \rangle$.
\item The order of $c_{j}$ is $n_{j}$ for $1\leq j\leq r$.
\item $\prod_{i=1}^{n} [a_{i} ,b_{i} ] \prod_{j=1}^{r} c_{j} = e_G$.
\end{enumerate}
\end{definition}

We note that this definition of a generating vector mimics the properties of a standard generating set for the fundamental group of a closed orientable surface.  For a discussion of the following Theorem, see for instance Broughton \cite{Bro1}.  

\begin{theorem}
\label{thm-setup2}
A finite group $G$ acts on a compact Riemann surface $X$ of genus $\sigma \geq 2$ with signature $(h;n_{1},\dots ,n_{r})$ if and only the Riemann-Hurwitz formula holds and there exists an $(h;n_{1},\dots ,n_{r})$-generating vector for $G$.
\end{theorem}

As one would expect, for an arbitrary signature $(h;n_1,\ldots, n_r)$ and an arbitrary finite group $G$, the general problem of determining whether or not there exists an $(h;n_1,\ldots, n_r)$-generating vector for $G$ is very difficult. Therefore, instead of attempting to enumerate group actions using generating vectors, we attack the potentially easier question of counting the number of {\em skeletal signatures} for a given genus $\sigma$, defined as follows.

\begin{definition}
An ordered pair $(h,r)$ of non-negative integers is {\em a skeletal signature for genus $\sigma \geq 2$} if there exists a compact Riemann surface $X$ of genus $\sigma$ and a finite group $G$ acting on $X$ with signature $(h;n_{1},\dots ,n_{r})$ for some $n_1,\ldots, n_r\ge 2$.  We denote the set of all skeletal signatures for genus $\sigma$ by $\mathcal{K}_\sigma$.
\end{definition}

We note that the actual orders of the branch points are not important for the definition of a skeletal signature.  As such, the collection of possible skeletal signatures corresponding to a given genus $\sigma$ provides a crude signature space, containing a part of the information carried by the space of all signatures, in a way that is more directly amenable to analysis.  We introduce skeletal signatures as an intermediate step in our counting problem, because directly attacking the question of counting all of the non-equivalent group actions on Riemann surfaces of a fixed genus $\sigma$, or even directly counting all of the possible signatures arising from such group actions, is at present an intractable problem.  

\section{Properties of $\mathcal{K}_\sigma$}
\label{properties}

In this Section, we consider some basic properties of $\mathcal{K}_\sigma$ for closed Riemann surfaces of genus $\sigma\ge 2$.   

We first note that the line with equation $r = -4h + 2\sigma + 2$ is naturally associated to the hyperelliptic involution.  We refer to this line as the {\em hyperelliptic line}.  Geometrically, the hyperelliptic involution can be viewed as taking the surface $X$ in ${\bf R}^3$ and arranging it so that there is an axis $L$ passing through all of the handles of the surface.  This axis intersects the surface in $2\sigma + 2$ points, with $2\sigma$ of the points coming from the passage of the axis through the $\sigma$ handles and the remaining 2 points being the extreme points of the intersection of $L$ with $X$.  Rotation by $\pi$ around $L$ yields a surface with genus $0$ and $2\sigma +2$ branch points of order 2 as a quotient.  By moving $2h$ handles off the axis in a way that is symmetric with respect to the involution, we obtain the action on $X$ for which the genus of the quotient is $h$ and the number of branch points is $2\sigma +2 - 4h$.

Define the triangular region $T_\sigma$ to be the region bounded by the axes $\{ h=0\}$ and$\{ r=0\}$, and the hyperelliptic line $\{ 2\sigma+2-4h = r\}$.    

\begin{lemma} The skeletal signature space $\mathcal{K}_\sigma$ is contained in $T_\sigma$.
\label{triangular}
\end{lemma}

\begin{proof}  We proceed naively.   Let $(h_0, r_0)$ be a point in $\mathcal{K}_\sigma$ arising from the signature $(h_0; n_1,\ldots, n_{r_0})$.  We recall the Riemann-Hurwitz formula:
\[ \sigma -1=|G|(h_0-1) +\frac{|G|}{2} \sum_{j=1}^{r_0} \bigg( 1-\frac{1}{n_j} \bigg) = |G| \bigg(h_0-1 +\frac{1}{2} \sum_{j=1}^{r_0} \bigg( 1-\frac{1}{n_j} \bigg) \bigg). \]
Note that the left hand side is fixed.  For a given $h_0$, we maximize $r_0$; this will give the highest potential skeletal signature on the vertical line $\{ h =h_0\}$. We can see that in order to maximize $r_0$, we need to maximize the number of terms in the sum, and hence minimize each term $1-\frac{1}{n_j}$ in the sum, and this minimum occurs when each $n_j = 2$.  This gives that the maximum value of $r_0$ satisfies
\[ \sigma -1= |G| \bigg( h_0-1 +\frac{1}{2} \sum_{j=1}^{r_0} \bigg( 1-\frac{1}{2} \bigg) \bigg) = |G| \bigg( h_0-1 +\frac{r_0}{4}  \bigg). \]
Since the product $|G| \bigg( h_0-1 +\frac{r_0}{4}  \bigg)$ is constant, we see that $r_0$ is maximized when $|G|$ is minimized, and the smallest possible value of $|G|$ is 2.  Hence, for a given value of $h_0$, the maximum value of $r_0$ satisfies $\sigma -1= 2 \bigg( h_0-1 +\frac{r_0}{4}  \bigg)$, which exactly yields the line $r = -4h + 2\sigma + 2$, as desired.
\end{proof}

This immediately gives the following upper bound on the number of points in $\mathcal{K}_\sigma$.

\begin{corollary} The number of points in $\mathcal{K}_\sigma$ for $\sigma\ge 2$ is at most quadratic in $\sigma$.
\label{quadratic (h,r) upper bound}
\end{corollary}

\begin{proof} We count the number of integer lattice points contained in $T_\sigma$.  By maximizing the genus of a possible quotient surface, we see that the rightmost skeletal signature $R$ in $\mathcal{K}_\sigma$ occurs either at $(\frac{1}{2}\sigma, 2)$ for $\sigma$ even or at $(\frac{1}{2}(\sigma +1), 0)$ for $\sigma$ odd.   Taking the appropriate upper limit for the outer sum (depending on the parity of $\sigma)$, we see that the number of skeletal signatures in $T_\sigma$ is 
\[ \sum_{h=0}^R \sum_{r=0}^{2\sigma+2-4h} 1 = \frac{1}{2}(\sigma+2)(\sigma+3).  \]
\end{proof}

The primary question of interest, given a point $(h_0,r_0)\in T_\sigma$, is whether $(h_0, r_0)$ lies in $\mathcal{K}_\sigma$; that is, whether or not $(h_0, r_0)$ is the skeletal signature of the action of some finite group $G$ on some Riemann surface $X$ of genus $\sigma$.  We note that this question is equivalent to asking whether there exists any group action on a compact Riemann surface of genus $\sigma$ with quotient of genus $h_0$ and where the natural branched cover is branched over $r_0$ points.  Hence, any reasonable analysis that allows us to exclude points from $\mathcal{K}_\sigma$ will thus allow us to exclude a large number of theoretically possible signatures and quotients.

There are a number of finer questions that follow from this primary question, such as whether a given point $(h_0, r_0)$ satisfies $(h_0,r_0)\in \mathcal{K}_\sigma$ for finitely or infinitely many $\sigma$, or even for all $\sigma$.  Before moving onto the proof of the lower bound on the size of $\mathcal{K}_\sigma$ in Section \ref{lower bound}, we discuss these finer questions.

To start, we observe that the the order of the group $G$ giving rise to a skeletal signature $(h_0, r_0)\in \mathcal{K}_\sigma$ is very roughly inversely proportional to the distance from $(h_0, r_0)$ to the origin $(0,0)$.   Given $N\ge 2$, let $L_{\sigma, N}$ be the triangular region bounded by the axes $\{ h=0\}$ and $\{ r=0\} $, and the line $\{ r=4\left( \frac{\sigma -1+N}{N} \right) -4h\}$.  Note that $L_{\sigma, N}\subset T_\sigma$ and in fact $L_{\sigma, 2}= T_\sigma$.

\begin{proposition} Fix a positive integer $N$ and a genus $\sigma\ge 2$. Then all skeletal signatures in $\mathcal{K}_\sigma$ for any group $G$ with $|G|\geq N$ lie in $L_{\sigma, N}$.
\label{prop-tri}
\end{proposition}

\begin{proof} Suppose that $(h_0,r_0)$ is a skeletal signature corresponding to a group $G$ with $|G|\ge N$ acting on a compact Riemann surface $X$ of genus $\sigma$ with signature $(h_0; n_{1},\dots ,n_{r_0})$. Applying the Riemann-Hurwitz formula and using the fact that $1-\frac{1}{n_{i}} \geq \frac{1}{2}$ (as $n_{i}\geq 2$), we see that 
\begin{eqnarray*}
\sigma -1 & = & |G|(h_0-1) +\frac{|G|}{2} \sum_{i=1}^{r_0} \left( 1-\frac{1}{n_{i}} \right) \\
& \geq & N(h_0-1) +\frac{N}{2} \sum_{i=1}^{r_0} \left( 1-\frac{1}{n_{i}} \right) \geq N(h_0-1) +\frac{r_0 N}{4}.
\end{eqnarray*}
Solving for $r_0$ gives 
$$4\left( \frac{\sigma -1+N}{N} \right) -4h_0 \geq r_0$$ 
and hence $(h_0,r_0)$ lies on or below the line $r=4\left( \frac{\sigma -1+N}{N} \right) -4h$.
\end{proof}

As a direct consequence of Proposition \ref{prop-tri}, we get the following further refinement of where the majority of skeletal signatures lie. 

\begin{corollary} For a fixed genus $\sigma \ge 2$, all points in $\mathcal{K}_\sigma$ lie on or below the line $r = \sigma + 2 - 3h$, with the exception of the point $(0, \sigma + 3)$ and the points $(h_0, 2\sigma+2-4h_0)$ (for $h_0\ge 0$) lying on the hyperelliptic line.
\label{gapcor}
\end{corollary}

\begin{proof}  Let $(h_0, r_0)$ be a point of $\mathcal{K}_\sigma$.  As in the proof of Lemma \ref{triangular}, if $|G| = 2$, then all branch points have order 2, and by the Riemann-Hurwitz formula, $h_0$ and $r_0$ satisfy the equation $r = 2\sigma +2 -4h$, which is the equation of the hyperelliptic line.  If $|G|=3$, then all branch points have order 3, and again by the Riemann-Hurwitz formula, $h_0$ and $r_0$ satisfy the equation $r = \sigma + 2 - 3h$.  

Suppose now that $|G|\ge 4$.  By Proposition \ref{prop-tri}, any skeletal signature $(h_0, r_0)$ for a group $G$ with $|G|\geq 4$ lies on or below the line $r=\sigma +3-4h$. The only point on this line when $h_0\geq 0$ which lies above the line $r = \sigma + 2 - 3h$ is the point $(0,\sigma +3)$. The result follows.
\end{proof}

We can be a bit more ambitious.   Recall the {\em Hurwitz bound}, that the order of the automorphism group of a closed Riemann surface of genus $\sigma\ge 2$ is at most $84(\sigma -1)$.  

Fix a number $0 < c < 1$ and consider the asymptotic question of determining the location in $\mathcal{K}_\sigma$ of the skeletal signatures corresponding to groups of order at most $c\cdot 84(\sigma -1)$ as $\sigma\rightarrow\infty$.  Applying Proposition \ref{prop-tri} infinitely many times with the values $N = c\cdot 84(\sigma -1)$ as $\sigma\rightarrow\infty$, we see that such skeletal signatures lie in the part of $\mathcal{K}_\sigma$ below the line $r = 4 +\frac{1}{21c} - 4h$.  

The interesting observation is that this line is independent of the genus $\sigma$.   For instance, if we take $c = \frac{1}{7}$, then the skeletal signatures corresponding to groups of order at least $\frac{1}{7}\cdot 84 (\sigma -1) = 12 (\sigma -1)$ lie in the triangular region bounded by the axes $\{ h=0 \}$ and $\{ r=0\}$, and the line $\{ r = \frac{13}{3} - 4h\}$.   The only skeletal signatures that lie in this region and that can occur (see Section \ref{persistently missing} below) are $(0,3)$ and $(0,4)$.  It follows that any group of order at least $12 (\sigma -1)$ yields a quotient with genus $0$ and either $3$ or $4$ branch points.

Another interesting value of $c$ in this discussion is $c = \frac{1}{21}$.  By a similar argument, this is the smallest value of $c$ for which the resulting triangular region contains a skeletal signature $(h_0, r_0)$ with $h_0\ge 1$.  Namely, for $c = \frac{1}{21}$, we see that the triangular region is bounded by the axes $\{ h=0 \}$ and $\{ r=0\}$, and the line $\{ r = 5 - 4h\}$, and this triangular region contains the point $(1,1)$.  (See Theorem \ref{sporadic theorem 1} below.)  Hence, the smallest order of the automorphism group of a surface of genus $\sigma$ for which the resulting quotient surface has genus at least 1 is $\frac{1}{21} \cdot 84(\sigma -1) = 4(\sigma -1)$.

These observations provide a geometric counterpoint to the standard algebraic derivations of similar results; see for instance Lemma 3.18 of Breuer \cite{Breu}.  We feel that this geometric counterpoint, making use of skeletal signatures, provides a new and interesting way of visualizing what had been previously largely algebraic derivations.  

In the following subsections, we consider different flavors of points that do and do not lie in $\mathcal{K}_\sigma$.  Our discussion of these points contains a fair bit of conjecture, which we gather together in Section \ref{conjectural}.  Our investigations, and the conjectural picture we develop for $\mathcal{K}_\sigma$, make extensive use of the {\bf genus} package develop by Breuer for the computer algebra system GAP \cite{GAP}; see also Breuer \cite{Breu}.  This package contains the details of all group actions on all closed Riemann surfaces of genus $2\le \sigma\le 48$.  

\subsection{Persistent points}
\label{persistent}

The point $(h_0,r_0)\in T_\sigma$ is {\em persistent} if there is $\sigma_0\ge 2$ so that $(h_0,r_0)\in \mathcal{K}_\sigma$ for all $\sigma\ge \sigma_0$, so that $(h_0, r_0)$ is a skeletal signature for all $\sigma\ge \sigma_0$.   (Such points can be defined either with the coordinates $h_0$ and $r_0$ given as functions of $\sigma$ or with coordinates being constants independent of $\sigma$.)   If we wish to keep track of the specific value of $\sigma_0$ beyond which a persistent point $(h_0, r_0)$ is always in $\mathcal{K}_\sigma$, we say that $(h_0, r_0)$ is {\em persistent for all $\sigma\ge \sigma_0$}.

One class of persistent points for all $\sigma\ge 2$ are those skeletal signatures lying on the hyperelliptic line, introduced in Section \ref{properties}.  One specific example is the point $(0,2\sigma+ 2)$ arising from the signature $(0;[2, 2\sigma + 2])$, which is the signature resulting from the complex structure on $X$ admitting the hyperelliptic involution; similarly, we have the skeletal signatures $(h_0, 2\sigma+2-4h_0)$ (for $h_0 \ge 0$) arising from the signatures $(h_0; [2,2\sigma+2-4h_0])$ of the other points lying on the hyperelliptic line.

A second class of persistent points for all $\sigma\ge 2$ are those skeletal signatures lying on the line $\{ r = \sigma + 2 - 3h\}$ corresponding to the actions of $C_3$ on compact Riemann surfaces of genus $\sigma$. Specifically, for a fixed genus $\sigma$, the group $C_{3}$ acts with skeletal signature $(h_0, \sigma +2-3h_0)$ for $h_0\geq 0$.  (We do note here that for $\sigma$ of the form $\sigma = 3k-1$ for $k\in {\bf N}$, the point $(\frac{1}{3}(\sigma +1), 1)$ does not lie in $\mathcal{K}_\sigma$; this is an immediate consequence of the fact that a necessary condition for the existence of an abelian group action (such as $C_3$) is that $r_0\ne 1$, since in this case commutators will be trivial.)

A third example of a persistent point for all $\sigma \ge 2$ is $(0, \sigma+3)$, corresponding to the signature $(0;[2,\sigma + 3])$, which comes from the $C_2\times C_2$ action on $X$ generated by the hyperelliptic involution and rotation by $\pi$ in an axis through the middle of $X$ orthogonal to the axis corresponding to the hyperelliptic involution.

An example of a persistent point whose coordinates are independent of genus is the point $(0, 3)$, which arises from any branched cover of the Riemann sphere by $X$ that is branched over 3 points.  It is well know that for every $\sigma\ge 2$, we can find a compact Riemann surface $X$ of genus $\sigma$ for which such a covering exists; see for instance Example 9.7 of Breuer \cite{Breu} in which an explicit example of such a surface is given for each $\sigma$.  Such surfaces are commonly known as quasiplatonic surfaces and arise in the study of dessins d'enfants.

Some persistent points arise from straightforward geometric realizations of cyclic automorphisms.

\begin{lemma} The point $(2,0)\in \mathcal{K}_\sigma$ for all $\sigma\ge 3$.
\label{(2,0) persistent}
\end{lemma}

\begin{proof} View the torus $T$ as the union of $\sigma -1$ parallel essential annuli $A_1,\ldots, A_{\sigma -1}$, and note that this description of $T$ naturally gives rise to a fixed point free action of $C_{\sigma -1}$ on $T$ by a rotation taking $A_j$ to $A_{j+1}$ (where $A_\sigma = A_1$).   Attach a handle to each $A_j$ in such a way that respects this rotation.  This yields a surface $X$ of genus $\sigma$ on which $C_{\sigma -1}$ acts without fixed points with a quotient of genus 2.  
\end{proof}

\begin{lemma} The point $(1,2)\in \mathcal{K}_\sigma$ for all $\sigma\ge 2$. 
\label{(1,2) persistent}
\end{lemma}

\begin{proof} View the 2-sphere $S$ as the union of $\sigma$ parallel bigons $B_1,\ldots, B_\sigma$, where the vertices of each $B_j$ are the north and south poles of $S$, and note that this description of $S$ naturally gives rise to an action of $C_\sigma$ on $S$ by a rotation fixing the north and south poles and taking $B_j$ to $B_{j+1}$ (where $B_{\sigma +1} = B_1$).  Attach a handle to each $B_j$ in such a way that respects this rotation.  This yields a surface $X$ of genus $\sigma$ on which $C_\sigma$ acts with quotient a surface with signature $(1; [2,\sigma])$, and hence a skeletal signature of $(1,2)$.  
\end{proof}

\subsection{Persistently missing points}
\label{persistently missing}

The point $(h_0,r_0)\in T_\sigma$ is {\em persistently missing} if there exists $\sigma_0\ge 2$ so that $(h_0,r_0)\not\in \mathcal{K}_\sigma$ for all $\sigma\ge \sigma_0$, so that $(h_0, r_0)$ is a skeletal signature for no $\sigma\ge \sigma_0$.    (As with persistent points, such points can be defined either with the coordinates $h_0$ and $r_0$ given as functions of $\sigma$ or with coordinates being constants independent of $\sigma$.)    As with persistent points, if we wish to keep track of the specific value of $\sigma_0$ beyond which a persistently missing point $(h_0, r_0)$ is never in $\mathcal{K}_\sigma$, we say that $(h_0, r_0)$ is {\em persistently missing for all $\sigma\ge \sigma_0$}.

For examples of persistently missing points, we see that the points $(0,0)$, $(0,1)$, $(0,2)$, and $(1,0)$ are all persistently missing points for all $\sigma\ge 2$, for the obvious reason that the surfaces with these signatures are not hyperbolic surfaces and so cannot be covered by a compact Riemann surface of genus $\sigma\ge 2$, even as a branched cover.  

We note that Corollary \ref{gapcor} can be interpreted as saying that every point $(h_0, r_0)$ lying strictly between the lines $\{ r = 2\sigma+2-4h \}$ and $\{ r = \sigma + 2 - 3h\}$ is persistently missing for all $\sigma\ge 2$, except for the point $(0, \sigma + 3)$ which is persistent for all $\sigma\ge 2$.

\subsection{Sporadic points}
\label{sporadic}

The point $(h_0,r_0)\in T_\sigma$ is {\em sporadic} if there are infinitely many genera $\sigma$ for which $(h_0,r_0)\in \mathcal{K}_\sigma$ and infinitely many genera $\sigma$ for which $(h_0,r_0)\not\in \mathcal{K}_\sigma$. (For sporadic points, we make the same distinction between those sporadic points whose coordinates are functions of $\sigma$, and those whose coordinates are independent of $\sigma$.)   

We have a complete picture of what occurs on the $h$-axis.  Specifically, we know from the discussion in Section \ref{persistently missing} that $(1,0)$ never occurs, for geometric considerations, while we know from Lemma \ref{(2,0) persistent} that $(2,0)\in \mathcal{K}_\sigma$ for all $\sigma\ge 2$.   

\begin{proposition} For each $h_0\ge 3$, the point $(h_0, 0)\in \mathcal{K}_\sigma$ if and only if $\frac{\sigma -1}{h_0 -1}\in {\bf N}$.  In particular, the point $(h_0, 0)$ is sporadic.
\label{r=0 sporadic}
\end{proposition}

\begin{proof} For $r_0 = 0$, the Riemann-Hurwitz formula reduces to the equation $\sigma - 1 = |G| (h_0 - 1)$.  In particular, the quantity $\frac{\sigma -1}{h_0 - 1}$ must be an integer.  Since $h_0 \ge 3$, there are infinitely many $\sigma$ for which $(h_0, 0)$ does not lie in $\mathcal{K}_\sigma$.

Suppose now that $\sigma = k(h_0 -1) + 1$.   Consider the surface of genus $\sigma$ formed as follows.  (This is very similar to the construction given in the proof of Lemma \ref{(2,0) persistent}.)   View the torus $T$ as the union of $k = \frac{\sigma -1}{h_0 -1}$ parallel essential annuli $A_1,\ldots, A_k$ , and note that this description of $T$ naturally gives rise to a fixed point free action of $C_k$ on $T$ by a rotation taking $A_j$ to $A_{j+1}$ (where $A_{k+1} = A_1$).   Attach a surface of genus $h_0 - 1$ to each $A_j$ in such a way that respects this rotation.  This yields a surface $S$ of genus $1 + k (h_0 -1) = \sigma$ on which $C_k$ acts without fixed points with a quotient of genus $h_0$.  

Hence, we see that $(h_0, 0)\in \mathcal{K}_\sigma$ if and only if $\frac{\sigma -1}{h_0 -1}\in {\bf N}$.  
\end{proof}

\begin{theorem} The point $(1,1)$ is a sporadic point. 
\label{sporadic theorem 1}
\end{theorem}

\begin{proof} First we shall show that $(1,1)$ is not a skeletal signature for any genus $\sigma = p+1$ where $p\geq 5$ is prime. Suppose to the contrary that $\sigma =p+1$ for some prime $p\geq 5$, and suppose that a group $G$ acts on a compact Riemann surface $X$ of genus $\sigma$ with signature $(1;n)$; that is, suppose that $(1,1)$ is a skeletal signature for $\sigma = p+1$.  Applying the Riemann-Hurwitz formula, we see that
$$p=\frac{|G|(n-1)}{2n} \text{ or } 2np=|G|(n-1).$$ 
Since $n$ and $n-1$ are relatively prime, it follows that $n-1$ divides $2p$ and thus we are in one of the four cases $n=2$ and $|G|=4p$; or $n=3$ and $|G|=3p$; or $n=p+1$ and $|G|=2(p+1)$; or $n=2p+1$ and $|G|=2p+1$. We  consider these cases separately.

First suppose that $n=2$ and $|G|=4p$, so $G$ acts with signature $(1;2)$.  Since $p\geq 5$, the Sylow Theorems imply that $G$ has a unique normal subgroup $H$ of index $4$ and order $p$. Applying a technical result due to Sah \cite{Sah} which allows us to determine the signature for $H$ given its index in $G$, the signature of $G$ and the orders of the elements of $G$ in the quotient group $G/H$, we see that no such $H$ can exist and thus $(1;2)$ is not a skeletal signature. We can apply a very similar argument for the case when $n=3$ and $|G|=3p$

For the remaining two cases, we first note that if $G$ acts with signature $(1;n)$ for some $n$, then  there exists a $(1;n)$-generating vector for $G$, or equivalently, three elements $a_{1}$, $b_{1}$ and $c_{1}$ that generate $G$ where $c_{1}$ is a commutator of $G$ of order $n$ (since $a_{1}b_{1}a_{1}^{-1}b_{1}^{-1} c_{1} =e_G$). Since $n\geq 2$, it follows that $G$ cannot be Abelian. Note that this implies the case when $n=2(p+1)$ and $|G|=2(p+1)$ cannot occur since $G$ would be cyclic. 

The remaining case to consider is when $n=p+1$ and $|G|=2(p+1)$, so $G$ acts with signature $(1;p+1)$. Since $p+1$ appears in the signature for $G$, we know that $G$ must contain an element of order $p+1$, and so it follows that $G$ has an index $2$ cyclic subgroup $H$. Since $H$ is cyclic, every subgroup of $H$ is characteristic and hence normal in $G$. Since $p\geq 5$, $p+1$ is even, so $H$ contains a subgroup $K$ of index $2$. Since $G/K$ has order $4$, it is Abelian, so it follows that the commutator subgroup of $G$ must be contained in $K$. However, $|K|=(p+1)/2$, so there do not exist any commutators of order $p+1$, and hence $(1;p+1)$ is not a skeletal signature for $G$.

To finish the proof, we shall construct an infinite sequence of $\sigma$ for which $(1,1)$ is a skeletal signature. Let $G_{n}=\langle x,y|x^n=y^2,y^{-1}xy=x^{-1} \rangle$, $n\geq 2$, denote the generalized quaternion group . Then the vector $(x,y,yx^{-2}y^{-1})$ is a $(1;n)$-generating vector for $G_{n}$. Applying the Riemann-Hurwitz formula, it follows that $G_{n}$ acts on a surface of genus $\sigma =2n-1$. In particular, $(1,1)$ is skeletal signature for $\sigma =2n-1$ for any integer $n\geq 2$.
\end{proof}

Unfortunately, we do not yet have a characterization of the specific values of $\sigma$ for which $(1,1)$ is and is not a skeletal signature.  We note here that the latter part of the proof can easily be adapted to show that all skeletal signatures of the form $(h_0, 1)$ occur for infinitely many $\sigma$.

\begin{lemma} For any $h_0\ge 2$, the point $(h_0, 1)\in \mathcal{K}_\sigma$ for infinitely many $\sigma\ge 2$. 
\label{(h,1) almost sporadic}
\end{lemma}

\begin{proof}  Following the argument given at the end of Theorem \ref{sporadic theorem 1}, and using that notation, we shall construct an infinite sequence of $\sigma$ for which $(h_0,1)$ is skeletal signature.  The vector $(x,y,e_{G_n}, \ldots, e_{G_n}, yx^{-2}y^{-1})$ is a $(1;n)$-generating vector for $G_{n}$, where there are $2(h_0 -1)$ instances of $e_{G_n}$. Applying the Riemann-Hurwitz formula, it follows that $G_{n}$ acts on a surface of genus $\sigma =4n(h_0 -1) + 2n-1$. In particular, $(h_0,1)$ is skeletal signature for $\sigma =4n(h_0 -1) + 2n-1$ for any integer $n\geq 2$.
\end{proof}

\subsection{Conjectural picture}
\label{conjectural}

In this Section, we augment the results above with a fairly complete conjectural picture of $\mathcal{K}_\sigma$.  We start by considering the lines $\{ h_0 = a \}$ for small values of $a\in {\bf N}\cup \{ 0\}$. 

\begin{conjecture}  The points $(0,r_0)$ for $4\le r_0\le \sigma + 2$ are persistent points for all $\sigma \ge 2$. 
\label{h=0 gaps}
\end{conjecture}

We have seen in Section \ref{persistent} that $(0, 2\sigma+2)$, $(0, \sigma+3)$, and $(0,3)$ are persistent points for all $\sigma \ge 2$, while Corollary \ref{gapcor} yields that no point strictly between $(0, 2\sigma+2)$ and $(0, \sigma+3)$ can be a skeletal signature.  Hence, combined with these results, Conjecture \ref{h=0 gaps} completes the description of all skeletal signatures of the form $(0, r_0)$.

A similar phenomenon occurs on the line $\{ h_0 = 1\}$.  

\begin{conjecture} The points $(1,r_0)$ for $3\le r_0\le \sigma-1$ are persistent points for all $\sigma\ge 2$.
\label{h=1 gap}
\end{conjecture}

Theorem \ref{sporadic theorem 1} largely describes the behavior of the point $(1,1)$, while Lemma \ref{(1,2) persistent} yields that $(1,2)$ is persistent for all $\sigma\ge 2$.  The discussion above and Corollary \ref{gapcor} show that that no point strictly between $(1, 2\sigma-2)$ and $(1, \sigma-1)$ can be a skeletal signature.  Hence, combined with these results, Conjecture \ref{h=1 gap} completes the description of all skeletal signatures of the form $(1, r_0)$. 

On the line $\{ h_0 = 2\}$, and indeed on $\{ h_0 = a\}$ for $a\ge 3$, the situation becomes more complicated.  Namely, we see experimentally that there are some persistent gaps, whose coordinates are dependent on $\sigma$, that occur in these lines.  Also, the behavior of a point $(h_0, r_0)$ for small $r_0$ becomes ragged.  We begin with the following conjecture.

\begin{conjecture} For $\sigma\ge 9$, let $E_\sigma$ be the line with slope $-3$ passing through $(1, \sigma -1)$ and let $D_\sigma$ be the line with slope $-4$ passing through $(1, \sigma -1)$.  Then no point strictly between $E_\sigma$ and $D_\sigma$ lies in $\mathcal{K}_\sigma$.
\label{triangular gap}
\end{conjecture}

Note that for $\sigma\le 8$, the set of points strictly between $E_\sigma$ and $D_\sigma$ is empty.  Corollary \ref{gapcor} and Conjecture \ref{triangular gap} describe an interesting phenomenon, namely that there are large parts of the triangular region $T_\sigma$ of potential skeletal signatures that in fact do not occur as skeletal signatures for any genus.  As skeletal signatures contain only the information about the number of branch points but not their specific orders, these gaps eliminate many potential signatures.  

We now turn our attention to the lines $\{ h_0 =2\}$ and $\{ h_0 = 3\}$.  For larger values of $a$, we get similar conjectural pictures, but unfortunately, we do not have enough evidence to formulate specific conjectures. 

\begin{conjecture} The point $(2,1)$ is sporadic.  The point $(2, (\frac{2}{3}\sigma -4))$ is persistently missing for all $\sigma\ge 7$.  All points $(2,r_0)$ for $2\le r_0 < (\frac{2}{3}\sigma -4)$ and $(\frac{2}{3}\sigma -4) < r_0 \le \sigma -4$ are persistent for all $\sigma\ge 7$, with the single exception that $(2,2)$ is not a skeletal signature for $\sigma=17$. 
\label{h=2}
\end{conjecture}

Together with the discussion above, Lemma \ref{(2,0) persistent}, Corollary \ref{gapcor} and Conjecture \ref{triangular gap}, Conjecture \ref{h=2} completes the description of all skeletal signatures of the form $(2, r_0)$.

\begin{conjecture} The point $(3,1)$ is sporadic.  The points $(3,(\frac{2}{3}\sigma - 7))$ and $(3,(\frac{2}{3}\sigma - 8))$ are persistently missing for all $\sigma\ge 18$.   For $\sigma\equiv 2 \mbox{(mod }3)$, the point $(3,(\frac{2}{3}\sigma - 6))$ is persistently missing for all $\sigma\ge 18$.  All remaining points $(3, r_0)$ with $2\le r_0\le \sigma -9$ are persistent for all $\sigma\ge \sigma_0$ for some $\sigma_0$.  
\label{h=3}
\end{conjecture}

Together with the discussion above, Lemma \ref{r=0 sporadic}, Corollary \ref{gapcor} and Conjecture \ref{triangular gap}, Conjecture \ref{h=3} completes the complete description of all skeletal signatures of the form $(3, r_0)$.  However, while we have a high level of confidence in this conjecture for the points $(3, r_0)$ for $r_0\ge 4$, the cases of $(3,2)$ and $(3,3)$ are more problematic.  While we feel that the evidence is suggestive for the behavior of these two skeletal signatures as $\sigma\rightarrow\infty$, we must recognize the possibility that one or the other, or both, are in fact sporadic.

Based on our analysis of the evidence to hand, including what we have been able to prove in previous Sections, we feel confident in making the following two strong conjectures, which when combined with the results from previous Sections provide a complete description of the behavior of any specific point $(h_0, r_0)$. 

\begin{conjecture} For any $h_0\ge 2$, the point $(h_0, 1)$ is sporadic.
\label{r=1 conjecture}
\end{conjecture}

However, we do not feel able to make a conjecture about the pattern of the values of $\sigma$ for which $(h_0, 1)$ is or is not a skeletal signature. This characterization is a subtle and difficult problem. 

\begin{conjecture} Any point $(h_0, r_0)$ with $h_0\ge 2$ and $r_0\ge 2$ is persistent for all $\sigma\ge \sigma_0$ for some $\sigma_0$.
\label{strong persistent conjecture}
\end{conjecture} 

\section{A Quadratic Lower Bound}
\label{lower bound}

To determine a lower bound for the number of distinct group actions on closed Riemann surfaces of genus $\sigma$, we shall determine the size of a subset of $\mathcal{K}_\sigma$ which corresponds to skeletal signatures for the action of the cyclic group $C_4$ of order 4 on $X$.

Fix $\sigma\ge 2$.  The procedure we follow in this Section is to first find all possible signatures for $C_4$ actions on $X$.  From the signatures, we find all skeletal signatures coming from $C_4$ actions.  We will then use transformations of signatures and the corresponding transformations of skeletal signatures to find the desired subset.

We begin by stating the following special case of a theorem of Harvey \cite{Har1} that we make extensive use of.

\begin{theorem} A signature $(h; [2,t_{1} ], [4,t_{2}])$ satisfying the Riemann-Hurwitz formula for genus $\sigma$ arises from a cyclic group $G$ of order $4$ acting on a closed orientable surface $X$ of genus $\sigma$ if and only if 
\begin{itemize}
\item for $h\neq 0$, $t_2$ is even; and
\item for $h=0$, $t_2 > 0$ and $t_2$ is even. 
\end{itemize}
\label{harvey theorem}
\end{theorem}

We define two operations on signatures.  The first operation $\mathcal{H}_1$ trades genus in the quotient for ramification points of order $2$. Define $\mathcal{H}_{1}$ by 
$$\mathcal{H}_{1} ( h; [2,t_1], [4,t_2]) =(h+1 ; [2, t_{1}-4], [4,t_2]),$$
assuming $t_1\ge 4$.   To see that $\mathcal{H}_1$ does indeed take signatures to signatures when $t_1\ge 4$, we note that a straightforward calculation shows that the Riemann-Hurwitz formula holds for $( h; [2,t_1], [4,t_2])$ if and only if it holds for $(h+1 ; [2, t_{1}-4, 4,t_2])$, and we then use Theorem \ref{harvey theorem} to see that the image signature $(h+1 ; [2, t_{1}-4],[ 4,t_2])$ is indeed a valid signature for a $C_4$ action on $X$. 

The second transformations trades ramification points of order 2 for ramification points of order 4.  Define $\mathcal{E}_{1,2}$ by 
$$\mathcal{E}_{1,2}( h; [2,t_1], [4,t_2]) = ( h; [2,t_1 - 3], [4,t_2+ 2]),$$
assuming $t_1\ge 3$.  As above, to see that $\mathcal{E}_{1,2}$ does indeed take signatures to signatures when $t_1\ge 3$, we note that a straightforward calculation shows that the Riemann-Hurwitz formula holds for $( h; [2,t_1], [4,t_2])$ if and only if it holds for $( h; [2,t_1 - 3], [4,t_2+ 2])$, and we then use Theorem \ref{harvey theorem} to see that the image signature $( h; [2,t_1 - 3], [4,t_2+ 2])$ is indeed a valid signature for a $C_4$ action on $X$. 

Note that these two operations on signatures descend to operations on skeletal signatures. Specifically, we have that $\mathcal{H}_{1} ((h_0,r_0) ) =(h_0+1,r_0-4)$ and $\mathcal{E}_{1,2} ((h_0,r_0) ) =(h_0,r_0-1)$.  We will use these two operations to construct a region in $\mathcal{K}_\sigma$ corresponding to $C_4$ actions on $X$.

We pause here to note that this discussion goes through for the action of any cyclic group $C_{p^2}$ for a prime $p$, and in fact for $C_n$ for any $n$, though the details become significantly more complicated in these cases.  However, we have only carried through the details for $C_4$, as this is sufficient for the purposes at hand.

Consider the following triangular subset of $T_\sigma$.  For a given genus $\sigma\ge 2$, we set
\[ k_\sigma = \lfloor \frac{\sigma}{3} \rfloor. \]
Let $S_\sigma$ be the triangle bounded by the lines $\{ r=-4h+\sigma+2 \}$ and $\{ r=-2h +\sigma+2-k_\sigma\}$, and the $r$-axis $\{ h=0\}$. 

\begin{lemma}
\label{lem-main} For any $\sigma\ge 12$, we have that $S_\sigma\subset \mathcal{K}_\sigma$.
\end{lemma}

\begin{proof} We first note that the signature $(0; [2, \sigma], [4,2])$ satisfies the criteria to be the signature of a $C_{4}$ action on some Riemann surface $X$ of genus $\sigma$, and this signature yields the skeletal signature $(0, \sigma+2) \in\mathcal{K}_\sigma$.     Applying $\mathcal{E}_{1,2}$ $k \ge 0$ times to the signature $(0; [2, \sigma], [4,2])$ results in the signature $(0; [2,\sigma-3k], [4,2+2k])$.  By Theorem \ref{harvey theorem}, this latter signature is a valid signature for a $C_4$ action on $X$ as long as $\sigma-3k \ge 0$, and so $k\le \frac{\sigma}{3}$, whence the definition of $k_\sigma = \lfloor \frac{\sigma}{3} \rfloor$.   Projecting the $k_\sigma +1$ signatures $(0; [2,\sigma-3k], [4,2+2k])$ for $0\le k\le k_\sigma$ yields the $k_\sigma +1$ skeletal signatures $(0,\sigma+2-k) \in \mathcal{K}_\sigma$ for $0\le k\le k_\sigma$.

Applying $\mathcal{H}_1$ $h$ times to the signature $(0; [2, \sigma], [4,2])$ results in the signature $(h; [2,\sigma -4h], [4,2])$.  This is a valid signature for a $C_4$ action on $X$ as long as $\sigma-4h\ge 0$, and when valid yields the skeletal signature $(h,\sigma +2-4h) \in \mathcal{K}_\sigma$.  We now apply $\mathcal{E}_{1,2}$ to $(h; [2,\sigma-4h], [4,2])$ $p\ge 0$ times, yielding the signature $(h; [2,\sigma-4h-3p], [4,2+2p])$.  Again by Theorem \ref{harvey theorem}, this is a valid signature for a $C_4$ action on some Riemann surface $X$ of genus $\sigma$ provided $\sigma-4h-3p \ge 0$. Assume that $p$ is chosen so that $(h; [2,\sigma-4h-3p], [4,2+2p])$ is a valid signature.

Note that the skeletal signature corresponding to $(h; [2,\sigma-4h-3p], [4,2+2p])$ is $(h,\sigma+2-4h-p)$.  This signature lies in $S_\sigma$ if and only if $\sigma+2-4h-p\ge \sigma+2-2h-k_\sigma$, which can be rewritten as $p\le -2h+k_\sigma$.  The condition that $(h,\sigma+2-4h-p)$ arises from a valid signature is that $p\le \frac{1}{3}(\sigma  -4h) \le k_\sigma -\frac{4}{3}h$.  Since we can apply $\mathcal{E}_{1,2}$ to the valid signature $(h; [2,\sigma-4h], [4,2])$ $p$ times, where $p\le k_\sigma -\frac{4}{3}h$, and still have a valid signature, we can certainly apply $\mathcal{E}_{1,2}$ $p$ times where $p\le k_\sigma -2h$, since $k_\sigma -2h < k_\sigma -\frac{4}{3}h$.  Hence, every integer lattice point on the vertical line segment between $(h,-4h+\sigma+2)$ and $(h,-2h +\sigma+2-k_\sigma)$ arises from a valid signature and therefore is a skeletal signature for $\sigma$, and so $S_\sigma\subset \mathcal{K}_\sigma$.
\end{proof}

We are now ready to prove our main result.

\begin{theorem} For $\sigma\ge 6$, there are at least $\frac{1}{4} (k_\sigma+1)(k_\sigma+3)$ distinct group actions on a closed Riemann surface of genus $\sigma$.
\label{thm-main} 
\end{theorem}

\begin{proof} Since $S_\sigma\subset\mathcal{K}_\sigma$, we need only count the number of points in $S_\sigma$, as each skeletal signature in $S_\sigma$ arises from the signature of the action of $C_4$ on a closed Riemann surface of genus $\sigma$ and different points in $S_\sigma$ necessarily correspond to distinct actions.

For $k_\sigma$ even, the number of integer lattice points in $S_\sigma$ is
\[ | S_\sigma| = \sum_{h=0}^{\frac{1}{2} k_\sigma} (-2h+1+k_\sigma) = \frac{1}{4} (k_\sigma +2)^2, \]
while for $k_\sigma$ odd, the number of integer lattice points in $S_\sigma$ is
\[ | S_\sigma| = \sum_{h=0}^{\frac{1}{2} (k_\sigma-1)} (-2h+1+k_\sigma) = \frac{1}{4} (k_\sigma+1)(k_\sigma+3) < \frac{1}{4} (k_\sigma +2)^2.  \]
\end{proof}

\end{document}